\documentclass[12pt,a4paper,leqno,oneside]{article}
\usepackage{amsmath}
\usepackage{amssymb}
\usepackage{amsthm}
\usepackage{psfrag}
\usepackage[utf8]{inputenc}
\usepackage{graphicx}
\usepackage[small]{caption}
\newtheorem{theorem}{Theorem}[section]
\newtheorem{corollary}[theorem]{Corollary}

\newtheorem{proposition}[theorem]{Proposition}
\newtheorem{lemma}[theorem]{Lemma}

\numberwithin{equation}{section}

\theoremstyle{remark}
\newtheorem{remark}[theorem]{Remark}

\newcommand{\vcon}{\stackrel{{\mathcal V}}{\longleftrightarrow}}

\def\supp{\mathop{\rm supp}\nolimits}

\def\plane{{\mathbb R}^2}
\def\P{{\mathbb P}}

\begin{document}

\title{Percolation in the vacant set of Poisson cylinders}
\author{Johan Tykesson\thanks{The Weizmann Institute of Science,
    Faculty of Mathematics and Computer Science, POB 26, Rehovot 76100, Israel. E-mail: {\tt
      johan.tykesson@gmail.com}. Research supported by a post-doctoral grant of the Swedish
     Research Council.} \and David Windisch\thanks{The Weizmann Institute of Science,
    Faculty of Mathematics and Computer Science, POB 26, Rehovot 76100, Israel. E-mail: {\tt
      d.windisch@gmail.com}.}}

\maketitle

\begin{abstract}
We consider a Poisson point process on the space of lines in ${\mathbb R}^d$, where a multiplicative factor $u>0$ of the intensity measure determines the density of lines. Each line in the process is taken as the axis of a bi-infinite cylinder of radius $1$. We investigate percolative properties of the vacant set, defined as the subset of ${\mathbb R}^d$ that is not covered by any such cylinder. We show that in dimensions $d \geq 4$, there is a critical value $u_*(d) \in (0,\infty)$, such that with probability $1$, the vacant set has an unbounded component if $u<u_*(d)$, and only bounded components if $u>u_*(d)$. For $d=3$, we prove that the vacant set does not percolate for large $u$ and that the vacant set intersected with a two-dimensional subspace of ${\mathbb R}^d$ does not even percolate for small $u>0$.
\end{abstract}

\section{Introduction} \label{s:intro}

If randomly selected straight lines are drilled through a large
ball, do the lines decompose the ball into small pieces? Or do the
lines only chop off small components, leaving the remaining object
essentially connected? This general consideration could serve as
an inspiration for this article. Random lines are defined by a
thickened Poisson line process in ${\mathbb R}^d$, $d \geq 3$, and
we study the connected components of the complement of this
process, called the vacant set. The intensity of lines is determined
by a parameter $u>0$ and the behavior of the vacant set is shown
to undergo a phase transition in $u$. Before we go into more
detail, let us introduce the model more precisely.

\medskip

On the space $\mathbb L$ of lines in ${\mathbb R}^d$, $d \geq 2$, there exists a non-trivial Haar measure $\mu$, invariant under translations and rotations, unique up to scaling. We consider a Poisson point process $\omega$ with intensity $u \mu$ on ${\mathbb L}$, where $u>0$, defined on a suitable probability space $(\Omega, {\mathcal A}, {\mathbb P}_u)$. Each line in the support of
$\omega$ is taken to be the axis of a bi-infinite closed
cylinder with base of radius $1$. Let ${\mathcal L}$ be the union of all
such cylinders, and define the \emph{vacant set}
$${\mathcal V}={\mathbb R}^d\setminus {\mathcal L}.$$
For details on the construction of $\mu$ and $\mathcal V$, we refer to the next section. The main question we investigate in this article is whether the random set ${\mathcal V}$ has unbounded connected components. If this is the case,
then we say that ${\mathcal V}$ \emph{percolates}.
Let $\theta(u)$ be the probability that the origin belongs to an unbounded connected component of ${\mathcal V}$. We will show in Proposition~\ref{p:0-1-perc} that $\mathcal V$ percolates with probability either $0$ or $1$, hence that $\theta(u)$ is positive precisely if ${\mathcal V}$ percolates almost surely. The main results in this article concern the non-degeneracy of the critical value
\begin{align}
 u_*(d) = \inf \{ u \geq 0 : \theta(u) = 0\} \in [0, \infty].
\end{align}
It is well known that $u_*(2)=0$,
see Theorem 10.3.2 in \cite{SW08}. Indeed, the
line process tiles the plane into finite-sided polygons, and several properties of these polygons have been studied in the literature, see
Section I.4.4 in \cite{S76}. From now on, we let $d \geq 3$ unless
otherwise specified. In Theorem~\ref{t:high}, we prove that in dimensions $d \geq 3$, $\mathcal V$ does not percolate for high intensities,
\begin{equation}\label{mtm1}
\text{for any $d \geq 3$, } u_*(d) < \infty,
\end{equation}
and in Theorem~\ref{lowintthm}, we prove that in dimensions $d \geq 4$, $\mathcal V$ does percolate at low intensities,
\begin{equation}\label{mtm2}
\text{for any $d \geq 4$, } u_*(d) > 0.
\end{equation}
This last result leaves as a natural open problem the case $d =3$. We observe in Proposition~\ref{p:triangles} that for $d=3$,
${\mathcal V}\cap {\mathbb R}^2$ does not even percolate for small $u>0$, although this does happen in higher dimensions, which presumably makes dimension $3$ substantially more difficult. Further open problems are mentioned in Remark~\ref{r:problems} at the end of the article.

\medskip

To a significant extent, this work is inspired by Sznitman's
recent work \cite{Szn09} on random interlacements, which share
some common features with Poisson cylinders. The process of random
interlacements is a site percolation model in ${\mathbb Z}^d$, $d
\geq 3$, corresponding to a Poisson point process on the space of
doubly-infinite trajectories on ${\mathbb Z}^d$ modulo time-shift.
As in the investigation of random interlacements undertaken in
\cite{SS10}, \cite{SS09}, \cite{Szn10_2}, \cite{ASz10},
\cite{Szn09} and \cite{T08}, the infinite-range dependence causes
some of the key difficulties in the present model and this makes
some of the techniques developed in \cite{Szn09} helpful.

\medskip

Before we comment on the proofs of the main results \eqref{mtm1} and
\eqref{mtm2}, let us note that the Poisson cylinder model
does not dominate and is not dominated by the standard
Poisson Boolean model of continuum percolation, see Remark~\ref{r:non-domination} ((1) and (3)) below
for details. Moreover, the model exhibits infinite-range
dependence. Indeed, for any
$x,y\in{\mathbb R}^d$ with $|x-y|>2$,
\begin{equation}
 \label{e:infrangedep}
\frac{c_{d,u}}{|x-y|^{d-1}} \, \leq \,  \textup{cov}_u
(\mathbf{1}_{x \in {\mathcal V}}, \mathbf{1}_{y \in {\mathcal V}})
\,  \leq \,  \frac{c_{d,u}'}{|x-y|^{d-1}},
\end{equation}
where $c_{d,u}>0$ and $c_{d,u}'>0$ are constants depending on $d$
and $u$ and $\textup{cov}_u$ denotes covariance with respect to ${\mathbb P}_u$ (see Remark~\ref{r:non-domination} (4)). In particular, this seriously complicates the use of Peierls-type
techniques in the present setup, and excludes the use of
arguments based on comparison with the Poisson Boolean model. We
instead use a renormalization scheme inspired by \cite{Szn09}.

\medskip

In order to prove \eqref{mtm1}, we consider a rapidly increasing
sequence $(a_n)_{n \geq 1}$ and the event that the ball of radius
$a_n$ centered at the origin is traversed by a vacant path
starting inside the ball of radius $1/4$ centered at the origin
and ending at the boundary of the ball of radius $a_n$. Since any
such path must cross the boundaries of the balls of radii $a_n/4$
and $a_n/2$ centered at the origin, this event is included in the
event that there exist two balls $B_1$ and $B_2$ of radius
$a_{n-1}$ centered at distance $a_n/4$ and $a_n/2$ from the origin
and traversed by a vacant path in the same way. If the occurrence
of such a traversal in $B_1$ was independent of the occurrence of
a traversal in $B_2$, then we would obtain a recursive inequality
of the form $p_n(u) \leq (const(d) a_n)^{2(d-1)} p_{n-1}(u)^2$ for
the probability $p_n(u)$ that a traversal occurs for a ball of
radius $a_n$, and iteration of this recursion would prove that
$\theta(u)=0$ for sufficiently large $u$ (here, the factor
$(const(d) a_n)^{2(d-1)}$ comes from the number of possible
choices of $B_1$ and $B_2$ and turns out not to be problematic by
a suitable choice of $a_n$). Of course, the local pictures left by
the process in the different neighborhoods $B_1$ and $B_2$ are in
fact not independent. The key observation at this point is that we
can make these local pictures independent by removing from the
Poisson cylinder process all cylinders intersecting both $B_1$ and
$B_2$. Such a removal is permitted, because we can thereby only
increase the value of $\theta(u)$. As a result, we do obtain a
recursion as mentioned above, but for the process with certain
cylinders removed. Iterating this procedure and bounding the total
loss of intensity from the successive removals, we can deduce that
$\theta(u)=0$ for sufficiently large $u$ and thus prove
\eqref{mtm1}.

\medskip

In order to prove \eqref{mtm2}, we prove that in dimension $d \geq
4$, even the set ${\mathcal V} \cap {\mathbb R}^2$ percolates for
$u>0$ chosen sufficiently small. This does not follow from a
standard application of a Peierls-type argument, because even in
high dimensions and even if $A$ is a subset of ${\mathbb R}$, the
probability ${\mathbb P}_u [A \subseteq {\mathcal L}]$ does not
decay exponentially in the length of $A$ (see
Remark~\ref{r:non-domination} (2) below). Instead, we use a
renormalization scheme similar to the one used in the proof of
\eqref{mtm1}, where instead of a vacant path we consider a circuit
in ${\mathcal L} \cap {\mathbb R}^2$ delimiting a bounded
component containing the origin in ${\mathbb R}^2$. Again, we
exclude cylinders from the process in order to obtain independence
between the local pictures in distant neighborhoods. In order to
obtain a useful bound on the probability that a cylinder occurs in
the excluded set, we need to assume that $d \geq 4$. In
Proposition~\ref{p:triangles}, we observe that in dimension $d=3$,
the set  ${\mathcal V} \cap {\mathbb R}^2$ indeed does not even
percolate for small $u>0$.

\medskip

The article is organized as follows: In Section \ref{s:notation}
we introduce the model and some of its basic properties. In
Section \ref{s:prel}, we provide useful preliminary estimates and
a $0$-$1$-law for the event that ${\mathcal V}$ percolates. In
Section~\ref{s:high}, we prove the main result \eqref{mtm1} for
the large $u$ regime and Section \ref{s:low} is devoted to the
main result \eqref{mtm2} pertaining to the small $u$ regime.
Section~\ref{s:low} also contains the proof that ${\mathcal
V}\cap{\mathbb R}^2$ does not percolate for any $u>0$ when $d=3$.

\medskip

Throughout the paper, we use the following convention concerning
constants: $c$ and $c'$ denote strictly positive constants only depending
on the dimension $d$, and their values are allowed to change from
place to place. The constants $c_0,c_1,..., c_{11}$ are fixed and
defined at their first place of appearance. Dependence of
constants on parameters other than $d$ will be indicated. For
example, $c(u)$ denotes a constant depending on $d$ and $u$.

\medskip

{\bf Acknowledgements.} We are grateful to Itai Benjamini for
introducing us to the random cylinder model, to Ofer Zeitouni for
helpful conversations, and to Gady Kozma and Alexander Holroyd,
whose intuition has led us to Proposition~\ref{p:triangles}. We also gratefully acknowledge
pertinent comments from an anonymous referee, who in particular has helped us shorten the proofs of Lemma~\ref{l:0-1} and Corollary~\ref{c:0-1}.

\section{Notation} \label{s:notation}
The goal of this section is to introduce the model of interest in
this paper.  First, we introduce some notation. The Euclidean norm on ${\mathbb R}^d$
will be denoted by $|\cdot|$. For sets $A, B \subseteq {\mathbb R}^d$, we define their mutual distance by
\begin{align}
d(A,B) = \inf \{ |x-y| : x \in A, y \in B\},
\end{align}
and their sum by
\begin{align}
A+B=\{ x+y : x\in A, y\in B\}.
\end{align}
For $A \subseteq {\mathbb R}^d$ and $t>0$, we define the $t$-neighborhood of $A$ as the open set
\begin{align}
 A^t = \{ x \in {\mathbb R}^d: d(\{x\},A) < t \},
\end{align}
and the $t$-interior of $A$ as the open set
\begin{align}
 A^{-t} = \{ x \in {\mathbb R}^d : d(\{x\},A^c) > t\}.
\end{align}
By $B(x,r)\subset {\mathbb R}^d$ we denote the open Euclidean
ball of radius $r>0$ centered at $x\in {\mathbb R}^d$ and by $B_\infty(x,r)$ the corresponding ball with respect to the $l_\infty$-norm.
For $x \in {\mathbb R}^{d-1}$, we let $B_{d-1}(x,r)$ be
the projection of $B((0,x),r)$ onto the plane orthogonal to $e_1$,
\begin{align}
B_{d-1}(x,r)=\{y\in {\mathbb R}^d : y_1=0,\,|y-x|<r\}.
\end{align}
For $x\ge 0$, let $\lceil x \rceil$ denote the smallest integer
that is larger than or equal to $x$, and let $[x]$ denote the integer part of $x$. For a topological space $X$,
we will write ${\mathcal B}(X)$ for the Borel-$\sigma$-algebra on
$X$. We will use the symbol $\pm$ to state two separate equations,
one with $-$ and one with $+$. For example, the expression
$\aleph^{\pm}=\beth^{\pm}$ stands for ``$\aleph^{-}=\beth^{-}$
and $\aleph^{+}=\beth^{+}$''.

\medskip

For the rigorous definition of the space of lines, we follow \cite{SW08}. The set ${\mathbb L}$ of lines is defined as the affine Grassmanian of $1$-dimensional affine subspaces of ${\mathbb R}^d$. In order to define a measure on $\mathbb L$, consider the group $SO_d$ of rotations of ${\mathbb R}^d$ about the origin, equipped with the topology induced from its representation as the group of orthogonal matrices with determinant $1$ and its unique Haar measure $\nu$, normalized such that $\nu(SO_d)=1$. For $x \in {\mathbb R}^d$, the translation $\tau_x$ is defined by
\begin{equation}
 \label{e:trans}
\begin{array}{cccc}
 \tau_x : & {\mathbb R}^d &\to & {\mathbb R}^d \\
& y & \mapsto & x+y
\end{array}
\end{equation}
Let $l_1 = \{ \alpha e_1 : \alpha \in {\mathbb R}\} \subset {\mathbb R}^d$ be the line along the vector $e_1$ of the canonical basis. We use $e_1^\perp$ to denote the $(d-1)$-dimensional subspace of ${\mathbb R}^d$ orthogonal to $e_1$, and $\lambda$ for the canonical Lebesgue measure on $e_1^\perp$. Then $\mathbb L$ is endowed with the finest topology such that the mapping
\begin{equation}
 \label{e:gamma}
\begin{array}{cccc}
 \gamma : & e_1^\perp \times SO_d &\to & {\mathbb L} \\
& (x,\vartheta) & \mapsto &  \vartheta (\tau_x(l_1))
\end{array}
\end{equation}
is continuous. One can thus define the Borel-$\sigma$-algebra ${\mathcal B}({\mathbb L})$ on $\mathbb L$. Up to constant multiples, the unique Haar measure $\mu$ on $({\mathbb L}, {\mathcal B}({\mathbb L}))$ (i.e. the unique nontrivial measure invariant under all translations $\tau_x$ and rotations $\vartheta$ of ${\mathbb R}^d$) is given by
\begin{align}
 \label{e:mu}
\mu = \gamma ( \lambda \otimes \nu ).
\end{align}
For background on Haar measures, we refer to Chapter~13 of \cite{SW08}, where a proof of this last statement is given in Theorem~13.2.12, p.~588. Next, we introduce the space $\Omega$ of locally finite point measures on $({\mathbb L}, {\mathcal B}({\mathbb L}))$:
\begin{equation}
\label{e:Omega}
\begin{split}
\Omega = \Big\{& \omega = \sum_{i \geq 0} \delta_{l_i}, \text{ with } l_i \in {\mathbb L}, \text{ and}\\
& \omega(A) < \infty \text{ whenever } A \text{ is compact } \Big\}.
\end{split}
\end{equation}
We endow $\Omega$ with the $\sigma$-algebra ${\mathcal A}$
generated by the evaluation maps $e_A: \omega \mapsto \omega(A)$,
for $A \in {\mathcal B}({\mathbb L})$. For $u>0$, we then define ${\mathbb
P}_u$ as the law on $(\Omega, {\mathcal A})$ of a Poisson point
process on $({\mathbb L}, {\mathcal B}({\mathbb L}))$ with intensity measure $u
\mu$. Let ${\mathbb E}_u$ denote the expectation operator
corresponding to ${\mathbb P}_u$.

\begin{remark}
\label{r:P-inv}
Translation and rotation invariance of $\mu$ imply that the law ${\mathbb P}_u$ on $\Omega$ of the Poisson point process on $\mathbb L$ is invariant under $(\tau_x)_{x \in {\mathbb R}^d}$ and $(\vartheta)_{\vartheta \in SO_d}$: Given $\omega = \sum_{i \geq 1} \delta_{l_i} \in \Omega$, define the $\tau_x$-image of $\omega$ as $\tau_x \omega = \sum_{i \geq 1} \delta_{\tau_x^{-1} l_i}$, and similarly for $\vartheta \omega$. Then ${\mathbb P}_u = \tau_x {\mathbb P}_u = \vartheta {\mathbb P}_u$.
\end{remark}

\medskip

We now introduce the events in $\mathcal A$ that will be of interest in this work. To any element $l \in {\mathbb L}$, we associate the cylinder
\begin{align}
 C (l) = l + \overline{B(0,1)} \subset {\mathbb R}^d.
\end{align}
For a subset $A$ of ${\mathbb R}^d$ that is either open or compact, we define the set of lines whose cylinder intersects $A$,
\begin{align}
\label{e:L_A}
 L_A = \{ l \in {\mathbb L} : C(l) \cap A \neq \emptyset \} = \{ l \in {\mathbb L} : l \cap (A+ \overline{B(0,1)}) \neq \emptyset \} .
\end{align}
If $A$ is open, then $L_A$ is an open and hence measurable subset of $\mathbb L$ and if $A$ is compact, then $L_A$ is the intersection of the open sets $L_{A^{1/n}}$, $n \geq 1$, hence again measurable. We also define the set of lines whose cylinders intersect both of the open or compact sets $A$ and $B$,
\begin{align}
L_{A,B} = L_A \cap L_B.
\end{align}
From the above definitions, we can compute $\mu(L_{B(x,r)})$:

\begin{lemma} \label{l:mubd0}
For $x \in {\mathbb R}^d$ and $r \geq 0$:
\begin{equation}
\label{e:nubd0}
  \mu (L_{B(x,r)}) = \lambda(B_{d-1}(0,r+1)).
\end{equation}
\end{lemma}

\begin{proof}
 By translation invariance of $\mu$, we can assume that $x=0$. Then the cylinder $C(l)$ intersects $B(0,r)$ if and only if $\gamma^{-1}(l) \in B_{d-1}(0,r+1) \times SO_d$, so the result follows from \eqref{e:mu} and the fact that $\nu (SO_d)=1$.
\end{proof}

We define the random set of cylinders
\begin{align}
 {\mathcal L} (\omega) = \bigcup_{l \in \supp \omega} C(l)  \subseteq {\mathbb
 R}^d,
\end{align} where $\supp \omega$ stands for the support of
$\omega$.
Since any $\omega \in \Omega$ is locally finite and \eqref{e:nubd0} holds, any closed ball $\overline{B(0,n)}$ is intersected by only finitely many cylinders $C(l), l \in \supp \omega$, so that $ {\mathcal L} (\omega) \cap  \overline{B(0,n)}$ is closed for all $n \geq 1$. In particular, ${\mathcal L} (\omega)$ is itself a closed set. The vacant set ${\mathcal V}(\omega)$ is defined as the open set
\begin{align}
 {\mathcal V}(\omega) = {\mathbb R}^d \setminus {\mathcal L} (\omega).
\end{align}
For sets $A, B \subseteq {\mathbb R}^d$, we will often consider events of the form
\begin{equation}
\label{e:vcon}
\begin{split}
 \{ A \vcon B \} = \{&\omega \in \Omega: \text{there is a continuous path} \\
&\text{from } A \text{ to } B \text{ contained in } {\mathcal V} (\omega) \},
\end{split}
\end{equation}
where a continuous path from $A$ to $B$ is a continuous function $p: [0,1] \to {\mathbb R}^d$, such that $p(0) \in A$ and $p(1) \in B$.
Let us check that the event $\{A \vcon B\}$ is indeed measurable.

\begin{lemma}
\label{l:meas}
For sets $A, B \subseteq {\mathbb R}^d$,
\begin{align}
\label{meas1} \{ A \vcon B\} \in {\mathcal A}.
\end{align}
\end{lemma}

\begin{proof}
In order to reduce the lemma to the case where $A$ and $B$ contain only one element, we begin with the following claim (recall that $e_.$ are the evaluation maps introduced below \eqref{e:Omega}):
\begin{equation}
\label{e:meas0}
\begin{split}
 \{ A \vcon B \} = \bigcup_{n \geq 1} \bigcup_{x \in A^{1/n} \cap {\mathbb Q}^d} & \bigcup_{y \in B^{1/n} \cap {\mathbb Q}^d} \{ \{x\} \vcon \{y\} \} \\
&\cap \{ e_{L_{B(x,1/n)}} = 0 \} \cap \{ e_{L_{B(y,1/n)}}=0\}.
\end{split}
\end{equation}
To see that the right-hand side in \eqref{e:meas0} is included in the left-hand side, observe that, if the right-hand side occurs, then there are points $x$ and $y$ in the $1/n$-neighborhoods of $A$ and $B$ connected by a vacant path. Moreover, thanks to the last two events on the right-hand side, $x$ and $y$ can be connected by vacant paths to $A$ and $B$ themselves. In particular, the event on the left-hand side occurs. On the other hand, if $A$ is connected to $B$ by a vacant path, then there are points $x_0 \in A \cap {\mathcal V}$ and $y_0 \in B \cap {\mathcal V}$ connected by a vacant path. In order to find $n, x$ and $y$ such that the right-hand side occurs, we choose $n \geq 1$ large enough such that $B(x_0,2/n)$ and $B(y_0,2/n)$ are subsets of the open set $\mathcal V$ and take any $x \in B(x_0,1/n) \cap {\mathbb Q}^d$ and $y \in B(y_0,1/n) \cap {\mathbb Q}^d$.

In order to prove that $\{ A \vcon B\}$ is measurable, it thus suffices to prove that $\{ \{x\} \vcon \{y\} \} \in {\mathcal A}$ for any $x,y \in {\mathbb Q}^d$. To this end, we define a \emph{rational polygonal path of $n$ steps} from $x$ to $y$ to be a continuous path $p$ from $x$ to $y$, determined by the values $$(p(0)=x, p(1/n),  \ldots, p((n-1)/n), p(1)=y) \in ({\mathbb Q}^{d})^{n+1}$$ and linear interpolation in between (i.e. $p((i/n) + \alpha) = (1- \alpha n)p(i/n) + \alpha n p((i+1)/n)$ for $0 \leq i \leq n-1$ and $0 \leq \alpha \leq 1/n$). Let ${\mathcal P}_{n,x,y}$ be the set of rational polygonal paths of $n$ steps from $x$ to $y$. Note that ${\mathcal P}_{n,x,y}$ is countable.
We claim that
\begin{align}
\label{e:meas}
\{ \{x\} \vcon \{y\}\} = \bigcup_{n \geq 0}  \bigcup_{p \in {\mathcal P}_{n,x,y}} \{ e_{L_{p([0,1])}}=0\}.
\end{align}
Indeed, if the event on the right-hand side occurs, then there exists a polygonal path from $x$ to $y$ that is not intersected by any of the cylinders and $\{\{x\} \vcon \{y\}\}$ must occur as well. If, on the other hand, $\{\{x\} \vcon \{y\}\}$ occurs, then there is some continuous path $p$ from $x$ to $y$ in the open set $\mathcal V$. Consider then a sequence $(p_n)_{n \geq 1}$ of rational polygonal paths of $n$ steps from $x$ to $y$ converging to $p$ in the supnorm. Since $p([0,1]) \subset {\mathcal V}$ and $\mathcal V$ is open, we must have $p_n([0,1]) \subset {\mathcal V}$ for $n$ large enough, implying that the event on the right-hand side in \eqref{e:meas} occurs. We have hence shown that \eqref{e:meas} holds. But since for any continuous path $p$, the set $p([0,1])$ is compact, the set $L_{p([0,1])}$ is in ${\mathcal B}({\mathbb L})$, as explained below \eqref{e:L_A}. Equation \eqref{e:meas} thus proves measurability of $\{\{x\} \vcon \{y\}\}$ and $\{ A \vcon B\}$.
\end{proof}

For any $x \in {\mathbb R}^d$, we can now define the event that $x$ belongs to an unbounded component of the vacant set,
\begin{align}
\label{e:vconinf}
 \{ x \vcon \infty \} = \bigcap_{n \geq 1} \big\{ \{x\} \vcon \partial B(0,n) \big\},
\end{align}
as well as the event that $\mathcal V$ percolates (i.e.~contains an unbounded component),
\begin{align}
\label{e:perc}
 \textup{Perc} =  \bigcup_{x \in {\mathbb Q}^d} \{x \vcon \infty\} = \bigcup_{x \in {\mathbb R}^d} \{x \vcon \infty\}.
\end{align}
We define the percolation probability
\begin{align}
\theta (u) = {\mathbb P}_u \big[ 0 \vcon \infty \big].
\end{align}
We will also consider the event that ${\mathcal V} \cap {\mathbb R}^2$ percolates, defined by
\begin{align}
\label{e:perc2}
 \textup{Perc}_2 = \bigcup_{x \in {\mathbb Q}^2 \subset {\mathbb R}^d} \{x \stackrel{{\mathcal V} \cap {\mathbb R}^2}{\longleftrightarrow} \infty\} = \bigcup_{x \in {\mathbb R}^2 \subset {\mathbb R}^d} \{x \stackrel{{\mathcal V} \cap {\mathbb R}^2}{\longleftrightarrow} \infty\},
\end{align}
where $\big\{ \{x\} \stackrel{{\mathcal V} \cap {\mathbb R}^2}{\longleftrightarrow} \infty \big\}$ is defined as in \eqref{e:vcon} and \eqref{e:vconinf} with $\mathcal V$ replaced by ${\mathcal V} \cap {\mathbb R}^2$ and we identify $\plane$ with the set of points $(x_1,...,x_d)\in{\mathbb R}^d$ for which $x_3=...=x_d=0$. Note that we have
\begin{align}
\label{e:perc2inc}
 \textup{Perc}_2 \subseteq \textup{Perc}.
\end{align}

\section{Preliminary results} \label{s:prel}

This section contains preliminary estimates, and a $0$-$1$-law for the events $\textup{Perc}$ and $\textup{Perc}_2$ introduced in \eqref{e:perc} and \eqref{e:perc2}. The following lemma provides an estimate on the measure of the set of cylinders intersecting two distant balls.
\begin{lemma} \label{l:nubd}
Consider open or closed balls $B_1 = B(x_1,r)$, $B_2=B(x_2,r) \subset {\mathbb R}^d$ such that $r \geq 0$ and $|x_1-x_2| = \alpha\ge 2(r+1)$. Then
\begin{align}
\label{e:nubd}
&c_1 \left( \frac{(r+1)^2}{\alpha} \right)^{d-1} \leq \mu(L_{B_1,B_2}) \leq c_2 \left( \frac{(r+1)^2}{\alpha} \right)^{d-1}.
\end{align}
\end{lemma}
\begin{proof}
By translation invariance of $\mu$, we can assume that $x_1=0$. We begin with the lower bound. Observe that
\begin{equation}\label{e.l31a}
\begin{split}
\gamma^{-1}(L_{B_1,B_2}) \supset \{ & (x,\vartheta)\,:\,x\in B_{d-1}(0,(r+1)/2),\, \\ & \vartheta({\mathbb R}_+ e_1)\cap B(x_2,(r+1)/2)\neq \emptyset\}.
\end{split}
\end{equation}
By rotational invariance of $\nu$, the function $\vartheta \in SO_d \mapsto \vartheta({\mathbb R}_+ e_1) \in \partial B(0,1)$ is a uniformly distributed random variable on the unit sphere under $\nu$. In particular, the image $m$ of $\nu$ under this function,
\begin{align}\label{e.l31b}
m(A)=\nu(\{\vartheta\,:\,\vartheta({\mathbb R}_+ e_1)\cap A\neq \emptyset\}), \text{ for } A \in {\mathcal B}({\mathbb R}^d),
\end{align}
is the uniform measure on the unit sphere.

From~\eqref{e.l31a}, \eqref{e.l31b} and \eqref{e:mu}, we obtain that
 \begin{align}\label{e.l31c}
 \mu(L_{B_1,B_2})\ge \lambda(B_{d-1}(0,(r+1)/2))m(B(x_2,(r+1)/2).
 \end{align}
 We proceed in a similar way for the upper bound. Observe that
 \begin{equation}\label{e.l31d}
\begin{split}
\gamma^{-1} & (L_{B_1,B_2})\subset \\
& \{(x,\vartheta)\,:\,x\in B_{d-1}(0,r+1),\, \vartheta({\mathbb R}_+ e_1)\cap B(x_2,2(r+1))\neq \emptyset\} \\ & \cup \{(x,\vartheta)\,:\,x\in B_{d-1}(0,r+1),\, \vartheta({\mathbb R}_- e_1)\cap B(x_2,2(r+1))\neq \emptyset\} .
\end{split}
\end{equation}
From~\eqref{e.l31b} and~\eqref{e.l31d} we obtain that
\begin{align}\label{e.l31e}
\mu(L_{B_1,B_2})\le 2 \lambda(B_{d-1}(0,r+1))m(B(x_2,2(r+1))).
\end{align}
The number of balls of radius $r+1$ centered at $\partial B(0,\alpha)$ needed to cover $\partial B(0,\alpha)$ is less than $c (\alpha/(r+1))^{d-1}$ and a subset of size at least $c' (\alpha/(r+1))^{d-1}$ of these balls consists of disjoint balls. Moreover, for each such ball $B$, $m(B)$ is the same. Hence,
\begin{equation}\label{e.l31g}
\begin{split}
c \left(\frac{r+1}{2\alpha}\right)^{d-1}\le m(B(x_2, & (r+1)/2))\\ & \le m(B(x_2,2 (r+1))) \le c' \left(\frac{2(r+1)}{\alpha}\right)^{d-1}.
\end{split}
\end{equation}
We now obtain the required bounds in~\eqref{e:nubd} using~\eqref{e.l31g}, ~\eqref{e.l31c}, ~\eqref{e.l31e} together with the estimates $\lambda(B_{d-1}(0,r+1))\le c (r+1)^{d-1}$ and $\lambda(B_{d-1}(0,(r+1)/2))\ge c (r+1)^{d-1}$.
\end{proof}

\begin{remark} \label{r:non-domination}
At this point, we can make some basic observations on the Poisson
cylinder process. In particular, we will highlight some essential differences between the Poisson cylinder process and the Poisson Boolean model consisting of a union of balls of fixed radius centered at the points of a Poisson point process in ${\mathbb R}^d$, see \cite{MeesterRoy}.
\begin{enumerate}
 \item The Poisson cylinder process is not dominated by the Poisson Boolean model, regardless of the choice of dimension and parameters. Indeed, let $l(x,y)$ be the straight line segment connecting two points $x$ and $y$ in ${\mathbb R}^d$, where $|x-y| > 2$. Then Lemma~\ref{l:nubd} implies the following lower bound on the probability that $l(x,y)$ is covered by $\mathcal L$:
\begin{equation}
\label{nond1}
\begin{split}
 {\mathbb P}_u [l(x,y) \subseteq {\mathcal L}] &\geq {\mathbb P}_u [ \omega \in \Omega: \omega(L_{\{x\},\{y\}}) >0] \\
& = 1 - e^{-u\mu(L_{\{x\},\{y\}})} \\
& \geq \frac{cu}{|x -y|^{d-1}}, \text{ for } u>0.
\end{split}
\end{equation}
In particular, this probability does not decay exponentially in
the distance between $x$ and $y$ as for the Poisson Boolean model
(cf. \cite{BJST09}, Lemma 3.4), showing that the Poisson cylinder
process is not dominated by the Poisson Boolean model with fixed radius balls. We do not know if it is impossible to dominate the Poisson cylinder process with the Poisson Boolean model with balls of some unbounded random radius.
 \item
The estimate \eqref{nond1} also implies that the probability ${\mathbb P}_u [A \subseteq {\mathcal L}]$ does not even decay exponentially in the length of $A$ if $A$ is restricted to be a subset of ${\mathbb R}$. This is a sharp contrast to random interlacements, where exponential decay of the corresponding probability does occur in dimension $d \geq 4$ if $A$ is restricted to be a subset of ${\mathbb Z}$ and at least in dimensions $d \geq 18$ for subsets $A$ of ${\mathbb Z}^2$ (see \cite{Szn09}, Theorem 2.4 for a generalized version of this assertion).
\item
The Poisson cylinder process does not dominate the Poisson Boolean
model, regardless of the choice of dimension and parameters.
Indeed, let $r>0$ and consider the probability that no cylinder
intersects the ball $B(0,r)$. By Lemma~\ref{l:mubd0}, we have
\begin{align*}
 {\mathbb P}_u [B(0,r) \subseteq {\mathcal V}] &= {\mathbb P}_u [ \omega \in \Omega: \omega(L_{B(0,r+1)}) =0] \\
&= e^{ - u \lambda (B_{d-1}(0,r+1))} \\
&\geq e^{-cur^{d-1}}, \text{ for } u>0.
\end{align*}
In particular, this probability does not decay exponentially in
the volume of $B(0,r)$, showing that the Poisson cylinder process
does not dominate the Poisson Boolean model.
\item
Lemma~\ref{l:nubd} also implies that the Poisson cylinder process exhibits long-range dependence: for any $x,y \in {\mathbb R}^d$ such that $|x-y| > 2$,
\begin{align}
 \label{e:cov0}
 c'ue^{-cu} \frac{1}{|x-y|^{d-1}} \, \leq \, \textup{cov}_u(\mathbf{1}_{x \in {\mathcal V}}, \mathbf{1}_{y \in {\mathcal V}}) \, \leq \, \frac{cu}{|x-y|^{d-1}},
\end{align}
where $\textup{cov}_u$ denotes covariance with respect to ${\mathbb P}_u$.
\end{enumerate}
\end{remark}

Using a similar procedure as in \cite{Szn09}, we prove in the following lemma that $(\tau_x)_{x \in {\mathbb R}^d}$ is ergodic.

\begin{lemma} \label{l:0-1}
For any $r>0$ and any measurable function $f : \Omega \to [0,1]$ satisfying $f(\omega)=f(\omega \mathbf{1}_{L_{B(0,r)}})$,
\begin{align}
\label{0-1:gen1} \Big| {\mathbb E}_u [f \, f\circ \tau_x] -
{\mathbb E}_u [f]^2 \Big| \leq cu \bigg( \frac{(r+1)^2}{|x|}
\bigg)^{d-1}.
\end{align}
Moreover, for any event $A \in {\mathcal A}$,
\begin{align}
 \label{0-1:gen}
\text{if $\tau_x(A)=A$ for all $x \in {\mathbb Q}  \subset {\mathbb R}^d$, ${\mathbb P}_u$-a.s., then ${\mathbb P}_u [A] \in \{0,1\}$.}
\end{align}
\end{lemma}

\begin{proof}
First, we observe that \eqref{0-1:gen} is a standard consequence of \eqref{0-1:gen1}.
Indeed, by $L^1({\mathbb P}_u)$-approximation of the indicator function of $A$ in \eqref{0-1:gen} by functions as in \eqref{0-1:gen1}, one deduces from \eqref{0-1:gen1} that ${\mathbb P}_u[A] = {\mathbb P}_u[A]^2$, hence \eqref{0-1:gen}. It thus only remains to prove \eqref{0-1:gen1}.

To this end, consider any $f$ as in the statement and for any $x \in {\mathbb R}^d$ such that $|x|>r$, let $g = f \circ \tau_x$. Note that then $g(\omega) = g(\omega \mathbf{1}_{B(x,r)})$. We write $f=f_1+f_2$, where $f_1 (\omega)=f(\omega \mathbf{1}_{L_{B(0,r)} \setminus L_{B(x,r)}})$ and $f_2=f-f_1$. Observe that the processes $\omega \mathbf{1}_{L_{B(0,r)} \setminus L_{B(x,r)}}$ and $\omega \mathbf{1}_{L_{B(x,r)}}$ are independent, so $f_1$ and $g$ are independent. Hence
\begin{equation}
\label{0-1:1}
\begin{split}
 {\mathbb E}_u[fg] &= {\mathbb E}_u [f_1] {\mathbb E}_u [g] + {\mathbb E}_u [f_2 g] \\
&={\mathbb E}_u[f] {\mathbb E}_u[g] - {\mathbb E}_u [f_2] {\mathbb E}_u [g] + {\mathbb E}_u [f_2 g].
\end{split}
\end{equation}
By translation invariance of $\omega$, we have ${\mathbb E}_u[g]={\mathbb E}_u[f]$. Using the bounds $|f_2 (\omega)| \leq \mathbf{1}_{\omega(L_{B(0,r),B(x,r)}) \geq 1}$, $|g| \leq 1$ and the Chebychev inequality, we thus deduce from \eqref{0-1:1} that
\begin{align*}
 \bigl| {\mathbb E}_u [f f \circ \tau_x] - {\mathbb E}_u[f]^2 \bigr| &\leq 2 {\mathbb P}_u [ \omega \in \Omega : \omega(L_{B(0,r),B(x,r)}) \geq 1 ]\\
& \leq 2 u \mu (L_{B(0,r),B(x,r)}).
\end{align*}
The estimate \eqref{0-1:gen1} thus follows from Lemma~\ref{l:nubd} and the proof of Lemma~\ref{l:0-1} is complete.
\end{proof}

The $0$-$1$-law for the events $\textup{Perc}$ and $\textup{Perc}_2$ introduced in \eqref{e:perc} and \eqref{e:perc2} can now be deduced from Lemma~\ref{l:0-1}.
\begin{corollary} \label{c:0-1}
For any $d \geq 3$ and $u>0$,
 \label{p:0-1-perc}
\begin{align}
\label{e:0-1}
&{\mathbb P}_u [ \textup{Perc} ] \in \{0,1\},\\
\label{e:0-1-perc2}
&{\mathbb P}_u [ \textup{Perc}_2 ]  \in \{0,1\}.
\end{align}
\end{corollary}

\begin{proof}
The assertions follow from \eqref{0-1:gen} and invariance of $\textup{Perc}$ and $\textup{Perc}_2$ under $(\tau_x)_{x \in {\mathbb Q}}$.
\end{proof}

\section{Absence of percolation for high intensities} \label{s:high}
The main objective of this section is to prove in Theorem~\ref{t:high} assertion~\eqref{mtm1} that in $3$ or higher dimensions, ${\mathcal V}$ does not percolate when $u$ is large enough. In the Poisson Boolean model, this can be proved by comparison with ordinary discrete percolation on the $d$-dimensional lattice (cf.~\cite{MeesterRoy}), but Remark~\ref{r:non-domination} shows that such an approach is impossible in the present case. Instead, we make use of a renormalization method, inspired by Section $3$ of~\cite{Szn09}. Although the sprinkling techniques developed in \cite{Szn09} for random interlacements cannot be used here due to the deterministic nature of cylinders, Poisson cylinders have the convenient property that any single cylinder occupies only a relatively small density of volume (order $n$ versus order $n^2$ for random interlacement paths, in a ball of radius $n$). This feature makes it easier to exclude sets of lines from the process without changing the percolative properties of the vacant set.

\begin{theorem}
\label{t:high}
For any dimension $d \geq 3$, there exists a constant $c \in (0,
\infty)$ such that $\theta(u)=0$ for all $u \geq c$.
\end{theorem}

\begin{proof}
Let $c_3$ be the volume of the $(d-1)$-dimensional ball of radius $3/4$:
\begin{align}
\label{e:vol}
c_3 = \lambda(B_{d-1}(0,3/4)).
\end{align}
For $c_2$ as on the right hand side of \eqref{e:nubd}, we then
define the number
\begin{align}
\label{e:a} a_0 = \left( 2 + 4 c_2 c_3^{-1} 5^{2(d-1)} \right)^{1/(d-1)} + 8^4,
\end{align}
as well as the rapidly growing sequence
\begin{align}
\label{e:an}
a_n = a_0^{(3/2)^n}, \, n \geq 0.
\end{align}
For any $n \geq 1$ that will remain fixed until the end of the proof, and $k = 1,2, \ldots, n$, we introduce the numbers
\begin{align}
\label{e:m}
m_k = c_4 a_{n-k+1}^{d-1},
\end{align}
where the dimension-dependent constant $c_4>1$ is chosen such that
\begin{equation}
\label{e:mkprop}
\begin{split}
&\text{the set $\partial B (0,a_{n-k+1}/4) \cup \partial B (0, a_{n-k+1}/2 )$ can be covered by $m_k$}\\
&\text{balls of radius $1/4$ centered in $\partial B (0,a_{n-k+1}/4) \cup \partial B (0, a_{n-k+1}/2 )$.}
\end{split}
\end{equation}
The aim now is to prove that $\theta(u)=0$, where
\begin{align}
\label{e:largelambda}
u = \frac{2}{c_3} \log \left(2c_4^2 a_0^{6(d-1)} \right),
\end{align}
which implies the theorem with $c=\left(2/c_3\right) \log
\left(2c_4^2 a_0^{6(d-1)}\right).$  For $1 \leq k \leq n$, we
define the level-$k$ index set by
\begin{align}
\label{e:ind}
{\mathcal I}_k = \{1,2,\ldots, m_1\} \times \{1,2,\ldots, m_2\} \times \cdots \times \{1,2,\ldots,m_k\}.
\end{align}
On ${\mathcal I}_{k+1}$, $1 \leq k \leq n-1$, we define the projection
$\pi_k : {\mathcal I}_{k+1} \to {\mathcal I}_k$ onto the first $k$
coordinates by $\pi_k ((i_1, \ldots, i_{k+1})) = (i_1, \ldots,
i_k)$. Moreover, for vectors $a$ and $b$, we let $(a,b)$ be
the concatenation of $a$ and $b$. Using \eqref{e:mkprop}, we can
now recursively define the array
$$(x^1_{v,i})_{v \in \{1,2\}, i \in {\mathcal I}_1},
(x^2_{v,i})_{v \in \{1,2\}^2, i \in {\mathcal I}_2}, \ldots,
(x^n_{v,i})_{v \in \{1,2\}^n, i \in {\mathcal I}_n},$$ such that
for $k=0,1, \ldots, n-1$, $i\in {\mathcal I}_k$, and $x^{0}_{.,.}
=0$ (here $\pi_0(i')=i$ is defined to hold for all $i' \in {\mathcal I}_1$),
\begin{equation}
\label{e:cov}
\begin{split}
&\partial B (x^k_{v,i},a_{n-k}/4) \subseteq \bigcup_{i' \in {\mathcal I}_{k+1}: \pi_k(i')=i} B(x^{k+1}_{(v,1),i'},1/4), \text{ and} \\
&\partial B (x^k_{v,i},a_{n-k}/2) \subseteq \bigcup_{i' \in {\mathcal I}_{k+1}: \pi_k(i')=i} B(x^{k+1}_{(v,2),i'},1/4).
\end{split}
\end{equation}
\begin{figure}[ht]
\psfrag{0}[cc][cc][1.5][0]{$x^k_{v,i}$}
\psfrag{an/2}[cc][cc][1.5][0]{$a_{n-k}/2$}
\psfrag{an/4}[cc][cc][1.5][0]{$a_{n-k}/4$}
\psfrag{an-1/2}[cc][cc][1.5][0]{$a_{n-k-1}/2$}
\psfrag{an-1/4}[cc][cc][1.5][0]{$a_{n-k-1}/4$}
\psfrag{p}[cc][cc][1.5][0]{$p$}
\begin{center}
\includegraphics[angle=0, width=0.7\textwidth]{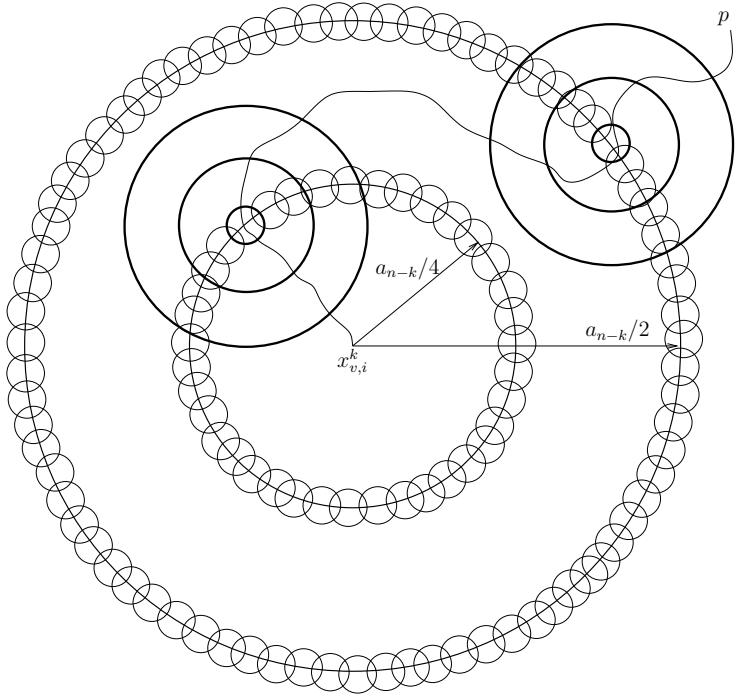}
\end{center}
\caption{Illustration of the construction of the array $((x^k_{v,i})_{v \in \{1,2\}^k, i \in {\mathcal I}_k})_{0 \leq k \leq n-1}$. The sphere of radius $a_{n-k}/4$ centered at $x^k_{v,i}$ is covered by balls of radius $1/4$ with centers $x^{k+1}_{(v,1),i'}$ and the sphere of radius $a_{n-k}/2$ is covered by balls with centers $x^{k+1}_{(v,2),i'}$. Each of the vertices $x^{k+1}_{(v,1),i'}$ and $x^{k+1}_{(v,2),i'}$ is then taken as the center of balls of radii $a_{n-k-1}/4$ and $a_{n-k-1}/2$, illustrated for the two small balls crossed by the path $p$.} \label{f:renormalization}
\end{figure}
This construction is illustrated in Figure~\ref{f:renormalization}. By \eqref{e:mkprop}, we can additionally choose the centers in \eqref{e:cov} such that $x^{k+1}_{(v,1),i'} \in \partial B(x^k_{v,i},a_{n-k}/4)$ and $x^{k+1}_{(v,2),i'} \in \partial B(x^k_{v,i},a_{n-k}/2)$. In particular, for any $i, i' \in {\mathcal I}_n$, $v, v' \in \{1,2\}^n$ and $1 \leq k \leq n$,
\begin{equation}
\label{e:Bdist}
\begin{split}
&\text{if } v_1=v'_1, v_2=v'_2, \ldots, v_{k-1}=v'_{k-1}, v_k \neq v'_k, \text{ then }\\
& |x^n_{v,i} - x^n_{v',i'}| \geq \frac{a_{n-k+1}}{4} - \sum_{l=k}^{n-1} a_{n-l} \geq \frac{a_{n-k+1}}{4} - (n-k) a_{n-k}\\
&\qquad \geq \frac{a_{n-k+1}}{4} - \frac{(\sqrt{a_0})^{(3/2)^{n-k}} a_{n-k}}{8} = \frac{a_{n-k+1}}{8}.
\end{split}
\end{equation}
For any $0 \leq k \leq n$, $v \in \{1,2\}^k$ and $i \in {\mathcal I}_k$, we define the associated event ${\mathcal C}^k_{v,i}$ as the event that there is a vacant path connecting $B(x^k_{v,i},1/4)$ to $\partial B(x^k_{v,i},a_{n-k})$:
\begin{equation*}
\begin{split}
{\mathcal C}^0 &= \left\{ B(0,1/4) \stackrel{\mathcal V}{\longleftrightarrow} \partial B(0,a_n) \right\}, \\
{\mathcal C}^k(v,i) &= \left\{ B(x^k_{v,i},1/4) \stackrel{\mathcal V}{\longleftrightarrow} \partial B(x^k_{v,i},a_{n-k}) \right\} \text{, for } 1 \leq k \leq n.
\end{split}
\end{equation*}
Note that, whatever the choice of $n \geq 1$, we have
\begin{align}
\label{e:cincl}
\{ 0 \vcon \infty \}  \subseteq {\mathcal C}^0.
\end{align}
Moreover, the properties \eqref{e:cov} allow us to deduce that, whenever ${\mathcal C}^0$ occurs, so must many of the events ${\mathcal C}^n(v,i)$. In order to state such an inclusion, we now introduce another index set
\begin{align}
\label{e:I+}
{\mathcal I}^+_k = \{1,2, \ldots, m_1\}^2 \times \{1,2, \ldots, m_2\}^{2^2} \times \cdots \times \{1,2,\ldots, m_k\}^{2^k},
\end{align}
for $1 \leq k \leq n$.
We denote elements of ${\mathcal I}^+_k$ by
\begin{align}
\label{e:I}
I = \left( (i_{v_1})_{v_1 \in \{1,2\}}, (i_{v_1,v_2})_{v_1,v_2 \in \{1,2\}}, \ldots, (i_{v_1,v_2, \ldots v_k})_{v_1, v_2, \ldots, v_k \in \{1,2\}} \right) \in {\mathcal I}^+_k,
\end{align}
where $i_{v_1, \ldots, v_l} \in \{1, \ldots, m_l\}$ for $1 \leq l \leq k$. For any $I \in {\mathcal I}^+_k$ and $v \in \{1,2\}^k$, we define, using the notation in \eqref{e:I},
\begin{align}
\label{e:Iv}
I_v = (i_{v_1}, i_{v_1,v_2}, \ldots, i_{v_1, v_2, \ldots, v_k}) \in {\mathcal I}_k.
\end{align}
The following lemma asserts that, for any $1 \leq k \leq n$, if ${\mathcal C}^0$ occurs, then there are numbers $i_1, i_2 \in \{1, \ldots, m_1\}$, $i_{1,1}, i_{1,2}, i_{2,1}, i_{2,2} \in \{1, \ldots, m_2\}$, \ldots, $(i_{v_1, \ldots, v_k} \in \{1, \ldots, m_k\})_{v_1, \ldots, v_k \in \{1,2\}}$, such that for any $v \in \{1,2\}^k$, the event ${\mathcal C}^k (v, I_v)$ occurs.

\begin{lemma}
\label{l:C}
For any $1 \leq k \leq n$,
\begin{equation}
\label{e:C}
\begin{split}
{\mathcal C}^0 \subseteq & \bigcup_{I \in {\mathcal I}^+_k} \bigcap_{v \in \{1,2\}^k} {\mathcal C}^k(v,I_v).
\end{split}
\end{equation}
\end{lemma}

\begin{proof}[Proof of Lemma~\ref{l:C}.]
We proceed by induction on $k$. For $k=1$, we observe that, if ${\mathcal C}^0$ occurs, then there is a continuous path $p$ in the vacant set connecting $B(0,1/4)$ to $\partial B(0,a_n)$. By continuity, $p$ must pass through $\partial B(0,a_n/4)$ and $\partial B(0,a_n/2)$. Hence, by \eqref{e:cov} with $k=0$, there are indices $i_1$ and $i_2$, such that $p$ passes through $B(x^1_{1,i_1},1/4)$ and $B(x^1_{2,i_2},1/4)$ (see Figure~\ref{f:renormalization}). Since $a_n \geq 2 a_{n-1}$ by \eqref{e:a} and \eqref{e:an}, this implies in particular that $p$ also visits $\partial B(x^1_{1,i_1},a_{n-1})$ and $\partial B(x^1_{2,i_2},a_{n-1})$. In other words, ${\mathcal C}^1 (1, i_1)$ and ${\mathcal C}^1 (2, i_2)$ both occur. This completes the proof for the case $k=1$.

Now assume that \eqref{e:C} holds for some $k < n$. Then numbers $$i_1, i_2, i_{1,1}, i_{1,2}, \ldots$$ with at most $k$ indices can be chosen as above such that the event ${\mathcal C}^k (v, I_v)$ occurs for any $v \in \{1,2\}^k$. Let us now fix any such $v$. Using the same argument as before, if ${\mathcal C}^k (v, I_v)$ occurs, then there exists a continuous vacant path $p$ connecting $B(x^k_{v,I_v}, 1/4)$ to $\partial B(x^k_{v,I_v}, a_{n-k})$, and by continuity and \eqref{e:cov}, this path must also connect $B(x^{k+1}_{(v,1),(I_v,i_{v,1})},1/4)$ to $\partial B(x^{k+1}_{(v,1),(I_v,i_{v,1})}, a_{n-k-1})$ and $B(x^{k+1}_{(v,2),(I_v,i_{v,2})},1/4)$ to $\partial B(x^{k+1}_{(v,2),(I_v,i_{v,2})}, a_{n-k-1})$ for appropriately chosen $i_{v,1}, i_{v,2}  \in \{1, \ldots, m_{k+1}\}$. In other words, for any $v \in \{1,2\}^k$, we can choose numbers $i_{v,1}, i_{v,2}  \in \{1, \ldots, m_{k+1}\}$ such that ${\mathcal C}^{k+1}((v,1),I_{(v,1)})$ and ${\mathcal C}^{k+1}((v,2),I_{(v,2)})$ occur. This proves the statement with $k$ replaced by $k+1$ and thus completes the proof of Lemma~\ref{l:C}.
\end{proof}

Let us now fix any $I \in {\mathcal I}^+$ as in \eqref{e:I}.
For any $v \in \{1,2\}^n$ we then define the set $L_v^- \subset {\mathbb L}$ as the set of lines whose cylinders intersect $B(x^n_{v,I_v},1/4)$ and at least one of the other balls $(B(x^n_{w,I_w},1/4))_{w \in \{1,2\}^n \setminus \{v\}}$:
\begin{align}
\label{e:L-}
L_v^- = \bigcup_{w \in \{1,2\}^n \setminus \{v\}} L_{B(x^n_{v,I_v},1/4), B(x^n_{w,I_w},1/4)},
\end{align}
and $L_v \subset {\mathbb L}$ as the set of lines whose cylinders intersect $B(x^n_{v,I_v},1/4)$, but none of the other balls $B(x^n_{w,I_w},1/4)_{w \in \{1,2\}^n \setminus \{u\}}$:
\begin{align}
\label{e:L} L_v = L_{B(x^n_{v,I_v},1/4)} \setminus L_v^- .
\end{align}
We now bound the intensity measure of $L_v^-$.

\begin{lemma}
\label{l:nuL-}
\begin{align}
\mu (L_v^-) \leq \frac{c_3}{2}.
\end{align}
\end{lemma}

\begin{proof}[Proof of Lemma~\ref{l:nuL-}.]
Suppose that $w \in \{1,2\}^n$ is different from $v$. If the first
coordinate at which $w$ and $v$ differ is $k \in \{1,\ldots, n\}$,
then there are at most $2^{n-k+1}$ possible choices for $w$ and by
\eqref{e:Bdist}, $|x^n_{v,I_v}-x^n_{w,I_w}| \geq a_{n-k+1}/8$. Hence we obtain from Lemma~\ref{l:nubd} (with $\alpha=a_{n-k+1}/8$
and $r=1/4$) that
{\allowdisplaybreaks
\begin{align*}
\mu (L_v^-) &\leq \sum_{k=1}^n 2^{n-k+1} c_2\left(\frac{(5/4)^2}{a_{n-k+1}/16} \right)^{d-1} \\
&= c_2 5^{2(d-1)} \sum_{l=1}^n  2^l \left( \frac{1}{a_0^{(3/2)^l}} \right)^{d-1}\\
&\leq c_2 5^{2(d-1)} \sum_{l=1}^\infty \left( \frac{2}{ a_0^{d-1}} \right)^l, \text{ using that } (3/2)^l \geq l, \\
&= \frac{2 c_2 5^{2(d-1)}}{a_0^{d-1}-2} \leq \frac{c_3}{2},
\text{ using } \eqref{e:a}. \qedhere
\end{align*}
}
\end{proof}
For $\omega \in \Omega$ and $L_v$ defined in \eqref{e:L}, we set
\begin{align}
\label{e:omegaui}
\omega_v = \omega \mathbf{1}_{L_v}, \text{ for } v \in \{1,2\}^n.
\end{align}
Since the sets $(L_v)_{v \in \{1,2\}^n}$ are disjoint, the processes $(\omega_v)_{v \in \{1,2\}^n}$ are independent and we have
\begin{align}
\label{e:decomp}
\sum_{v \in \{1,2\}^n } \omega_v \leq \omega .
\end{align}
In particular, if we define the event
\begin{align*}
{\tilde {\mathcal C}}^n_v = \left\{ \omega \in \Omega:
B(x^n_{v,I_v},1/4) \stackrel{{\mathcal
V}(\omega_v)}{\longleftrightarrow} \partial B(x^n_{v,I_v},a_0)
\right\},
\end{align*}
then the events $( {\tilde {\mathcal C}}^n_v)_{v \in \{1,2\}^n}$ are independent and
\begin{align*}
{\mathcal C}^n(v,I_v) \subseteq {\tilde {\mathcal C}}^n_v, \text{ for } v \in \{1,2\}^n.
\end{align*}
Hence, we obtain from Lemma~\ref{l:C}, the fact that the cardinality of ${\mathcal
I}_k^{+}$ equals $\prod_{i=1}^{k}m_i^{2^i}$, and the union bound that
\begin{align*}
{\mathbb P}_u [ {\mathcal C}^0 ] \leq m_1^2 m_2^{2^2} \cdots m_n^{2^n} \sup_{v \in \{1,2\}^n} {\mathbb P}_u [ {\tilde {\mathcal C}}^n_v ]^{2^n}.
\end{align*}
With \eqref{e:an} and \eqref{e:m}, we hence find that
\begin{equation}
\label{e:punch1}
\begin{split}
{\mathbb P}_u [ {\mathcal C}^0 ] &\leq c_4^{2+2^2 + \cdots + 2^n} a_0^{(d-1) 2^{n+1} \sum_{1 \leq k \leq n} (3/4)^k} \sup_{v \in \{1,2\}^n} {\mathbb P}_u [ {\tilde {\mathcal C}}^n_v ]^{2^n} \\
&\leq \left( c_4^2 a_0^{6(d-1)} \right)^{2^n} \sup_{v \in \{1,2\}^n} {\mathbb P}_u [ {\tilde {\mathcal C}}^n_v ]^{2^n}.
\end{split}
\end{equation}
For any $v \in \{1,2\}^n$, let now ${\bar L}_v \subset {\mathbb L}$ be the set of lines whose cylinder covers the ball $B(x^n_{v,I_v},1/4)$,
\begin{align*}
{\bar L}_v = \left\{ l \in {\mathbb L}: C(l) \supset B(x^n_{v,I_v},1/4) \right\}.
\end{align*}
In order to compute $\mu({\bar L}_v)$, we use translation invariance of ${\mathbb P}_u$ to replace $x^n_{v,I_v}$ by $0$, then observe that $\gamma(x,\vartheta)\in {\bar L}_v$ if and only if $(x,\vartheta)\in B_{d-1}(0,3/4)\times SO_d$.
It follows that $\mu \left({\bar L}_v \right) = \lambda \left( B_{d-1}(0,3/4) \right) = c_3,$ so by Lemma~\ref{l:nuL-},
\begin{equation}
\label{e:nurem}
 \mu \left( {\bar L}_v \setminus L^-_v \right) \geq \frac{c_3}{2}.
\end{equation}
Now observe that if ${\tilde {\mathcal C}}^n_v$ occurs, then no line in ${\bar L}_v \setminus L^-_v$ can appear in the process, i.e.
\begin{align*}
 {\tilde {\mathcal C}}^n_v \subseteq \left\{\omega \in \Omega: \omega_v ({\bar L}_v \setminus L^-_v) = 0 \right\} = \left\{\omega \in \Omega: \omega ({\bar L}_v \setminus L^-_v) = 0 \right\},
\end{align*}
using that ${\bar L}_v \setminus L^-_v \subseteq L_v$ for the equality.
Hence, for any $v \in \{1,2\}^n$,
\begin{align*}
 {\mathbb P}_u [  {\tilde {\mathcal C}}^n_v ] \leq {\mathbb P}_u \left[  \omega ({\bar L}_v \setminus L^-_v) = 0  \right] = e^{-u \mu \left( {\bar L}_v \setminus L^-_v \right)} \leq e^{-u c_3/2}, \text{ by } \eqref{e:nurem}.
\end{align*}
Inserting this bound into \eqref{e:punch1}, we deduce with \eqref{e:cincl} that for all $n \geq 1$,
\begin{align*}
\theta(u) \leq {\mathbb P}_u \left[ {\mathcal C}^0 \right] &\leq \left( c_4^2 a_0^{6(d-1)} \right)^{2^n} e^{- u 2^n c_3/2} \stackrel{\eqref{e:largelambda}}{=} 2^{-2^n}.
\end{align*}
Letting $n$ tend to infinity, we obtain Theorem~\ref{t:high}.
\end{proof}

\section{Percolation for low intensities}\label{s:low}

In this section we prove assertion \eqref{mtm2} stating that for any $d\ge 4$ and
$u>0$ chosen small enough, ${\mathcal V}$ percolates
almost surely. In fact, we will show the stronger statement that even ${\mathcal V} \cap {\mathbb R}^2$ percolates in this regime (cf.~\eqref{e:perc2}, \eqref{e:perc2inc}). Again, we will use a renormalization scheme inspired by \cite{Szn09}. We also remark that there are some similarities between our proof and the methods used in \cite{JBG08} to prove that there is no percolation at low intensities in the covered set of the Poisson Boolean model with balls of unbounded radii that satisfy a moment condition.  The difficulties prohibiting the use of standard percolation techniques have been mentioned in Remark~\ref{r:non-domination} (2) above.

\begin{theorem}\label{lowintthm}
For any $d\ge 4$, there exists $c> 0$ such that
\begin{equation}
\P_{u}[\textup{Perc}_2]=1, \text{ for } u \in [0,c]
\end{equation}
\end{theorem}

The proof of Theorem~\ref{lowintthm} will follow from Lemma~\ref{squarest}, Proposition~\ref{keycontrol} and Lemma~\ref{indlemma} below, with the key control appearing in Proposition~\ref{keycontrol}.
We identify $\plane$ with the set of points $(x_1,...,x_d)\in{\mathbb
R}^d$ for which $x_3=...=x_d=0$ and ${\mathbb R}$ with the set of
points $(x_1,...,x_d)\in{\mathbb R}^d$ for which $x_2=...=x_d=0$.
In this section, we define
\begin{equation}
\label{consts}
a_0 \in [c,\infty)\mbox{ and }
c_5= \frac{3.8}{3},
\end{equation}
where the dimension-dependent constant $c = (2c_8)^{10}>0$ will be chosen later. Let
\begin{equation}
\label{rdef}a_n=a_0^{c_5^n}\mbox{, }n\ge 1.
\end{equation}
In particular, we have
\begin{equation}\label{req1}
a_n=a_{n-1}^{c_5},\mbox{ }n\ge 1.
\end{equation}
Let $S(x,t)\subset \plane$ be the $2$-dimensional closed
$l_{\infty}$-ball of radius $t\ge 0$, centered at $x\in \plane$.
For $n\ge 0$, $x,y\in \plane$, let $L_{x,y,n}^-\subset{\mathbb
L}$ be the set of lines whose associated cylinders intersect
$S(x,a_{n})$ as well as $S(y,a_{n})$:
\begin{align}
\label{L-}
L_{x,y,n}^-=L_{S(x,a_{n}),S(y,a_{n})}.
\end{align}
Moreover, for $x,y\in \plane$, $n\ge 0$, let
$L_{x,y,n}^+\subset{\mathbb L}$ be the set of lines whose
cylinders intersect $S(x,a_n)$, but not $S(y,a_n)$:
\begin{align}
 \label{L+}
L^+_{x,y,n}=L_{S(x,a_n)} \setminus L_{x,y,n}^-.
\end{align}
For $n\ge
0$, $x\in \plane$ and $L \in {\mathcal B}({\mathbb L})$, we define the event
\begin{align*}
A_n(x,L)= \Big\{ &\omega \in \Omega: S(x,a_n/10) \text{ is connected to } \partial
S(x,a_n)\\
& \text{ by a path in } {\mathcal
L}(\omega \mathbf{1}_L)\cap {\mathbb R}^2 \Big\},
\end{align*}
and set
\begin{align*}
 A_n(x) = A_n(x, {\mathbb L}).
\end{align*}

Let us check that the event $A_n(x,L)$ is indeed measurable.

\begin{lemma}For $n\ge
0$, $x\in \plane$ and $L \in {\mathcal B}({\mathbb L})$,
\label{l:Ameas}
 \begin{align}
  \label{e:Ameas}
A_n(x,L) \in {\mathcal A}.
 \end{align}
\end{lemma}

\begin{proof}
By local finiteness of $\omega$, the set $({\mathcal L}(\omega \mathbf{1}_L) \cap S(x,a_n)) \cup S(x,a_n/10)$ is a union of finitely many closed convex sets. Any two of these sets intersect each other if and only if their open $1/m$-neighborhoods intersect each other for all $m \geq 1$. Hence,
\begin{align*}
  A_n(x,L) = \bigcap_{m \geq 1} \Big\{ &\omega \in \Omega: S(x,a_n/10)^{1/m} \text{ is connected to }
(S(x,a_n)^c)^{1/m}\\
& \text{ by a path in } {\mathcal
L}(\omega \mathbf{1}_L)^{1/m} \cap {\mathbb R}^2 \Big\},
\end{align*}
Since ${\mathcal L}(\omega \mathbf{1}_L)^{1/m} \cup S(x,a_n/10)^{1/m}$ is a union of open convex sets, a path from $S(x,a_n/10)^{1/m}$ to $(S(x,a_n)^c)^{1/m}$ in ${\mathcal L}(\omega \mathbf{1}_L)^{1/m} \cap {\mathbb R}^2$ exists if and only if there exists a rational polygonal path (see Lemma~\ref{l:meas} for the definition) in ${\mathcal L}(\omega \mathbf{1}_L)^{1/m} \cap {\mathbb R}^2$ from $S(x,a_n/10)^{1/m}$ to $(S(x,a_n)^c)^{1/m}$, such that every segment of the path is covered by the $1/m$-neighborhood of the intersection of a cylinder with ${\mathbb R}^2$. Let ${\mathcal P}_{k,z_1,z_2}$ be the countable set of rational polygonal paths consisting of $k$ segments from $z_1 \in {\mathbb Q}^2$ to $z_2 \in {\mathbb Q}^2$ in ${\mathbb R}^2$. For any $p \in {\mathcal P}_{k,z_1,z_2}$, let $p_1, \ldots, p_k$ be its $k$ segments, and let $C(p_i)$ be the open set of lines $l$ such that $p_i \subseteq (C(l)\cap {\mathbb R}^2)^{1/m}$. Then the remarks above imply that
\begin{align*}
 A_n(x,L) = &\bigcap_{m \geq 1} \bigcup_{z_1 \in S(x,a_n/10)^{1/m} \cap {\mathbb Q}^2} \bigcup_{z_2 \in (S(x,a_n)^c)^{1/m} \cap {\mathbb Q}^2} \\
&\bigcup_{k \geq 1} \bigcup_{p \in {\mathcal P}_{k,z_1,z_2}} \bigcap_{i=1}^k \{ \omega \in \Omega: e_{C(p_i)}(\omega \mathbf{1}_L) \geq 1 \},
\end{align*}
proving Lemma~\ref{l:Ameas}.
\end{proof}

By translation invariance of $\P_u$ (cf.~Remark~\ref{r:P-inv}),
\begin{equation}\label{atransinv}\P_u[A_n(x)]\mbox{ is
independent of }x,\end{equation} and we put
\begin{align}
 \label{e:pnu}
p_n(u)=\P_{u}[A_n(0)],\mbox{ }n\ge
0,\mbox{ }u\ge 0.
\end{align}
Moreover, for $L_1, L_2 \in {\mathcal B}({\mathbb L})$,
\begin{equation}\label{aincr}
\text{if } L_1 \subseteq L_2, \text{ then } A_n(x,L_1) \subseteq A_n(x,L_2).
\end{equation}
For $n\ge 1$, let
${\mathcal H}_n^1$ be a minimal collection of points on
$\partial S(0,a_n/2)$ such that
$$\partial S(0,a_n/2)\subset \bigcup_{x\in {\mathcal H}_n^1}
S(x,a_{n-1}/10),$$ and let ${\mathcal H}_n^2$ be a minimal
collection of points on $\partial S(0,3 a_n/4)$ such that
$$\partial
S(0,3 a_n/4)\subset \bigcup_{x\in {\mathcal H}_n^2}
S(x,a_{n-1}/10).$$
Note that we then have
\begin{align}
\label{e:|H|}
|{\mathcal H}_n^1| + |{\mathcal H}_n^2| \leq c \frac{a_n}{a_{n-1}}.
\end{align}
We need a preliminary lemma.
\begin{lemma}[$d \geq 3$]
\label{squarest}
Consider two (open or closed) disjoint squares $C_1$, $C_2 \subset
{\mathbb R}^2$ of side length $s\ge 1$ such that $d(C_1,C_2) = r\ge 4$. Then
\begin{align}
\mu(L_{C_1,C_2}) \leq c \frac{s^2}{r^{d-1}}.
\end{align}
\end{lemma}
\begin{proof}
Let $C_1$ and $C_2$ be two squares fulfilling the assumptions of
the lemma. For $i=1,2$, let ${\mathcal U}_i$ be a minimal
collection of balls of radius $1$ centered on $\partial C_i$ such
that $\partial C_i\subset \cup_{B\in {\mathcal U}_i}B$. By
continuity, any cylinder that intersects both $C_1$ and $C_2$ must
also intersect both $\partial C_1$ and $\partial C_2$. Moreover,
any such cylinder intersects at least one ball $B\in {\mathcal
U}_1$ and at least one ball $B'\in {\mathcal U}_2$. Therefore,
using a union bound and Lemma~\ref{l:nubd}, we get
\begin{align}
\mu(L_{C_1,C_2})\le \sum_{B\in {\mathcal U}_1} \sum_{B'\in
{\mathcal U}_2}\mu(L_{B,B'})\le \sum_{B\in {\mathcal U}_1}
\sum_{B'\in {\mathcal U}_2} c r^{-(d-1)}\le c' s^2 r^{-(d-1)}.
\end{align}
Here, the third inequality uses $|{\mathcal U}_i|\le c s$, for $i=1,2$.
\end{proof}

The key estimate for the proof of Theorem~\ref{lowintthm} is
provided by the following control on the probabilities $p_n(u)$ defined in \eqref{e:pnu}. 
We can only prove \eqref{keydec} for $d\ge 4$. This is not merely a deficiency of the arguments; we shall see later in Proposition~\ref{p:triangles} that $p_n(u)$ does not decay for $d=3$.

\begin{proposition}[$d \geq 4$]
\label{keycontrol}
There exists a constant $c_6 \in (0,\infty)$ such that for any $a_0 \ge c_6$ and
$u \in (0, c_7(a_0))$ one has
\begin{equation}\label{keydec}
p_n(u)\le a_n^{(2-c_5(d-1))/2},\mbox{ }n\ge 1.
\end{equation}
\end{proposition}

\begin{proof}
Observe that if $A_n(0)$ occurs, then for at least one $x_1\in
{\mathcal H}_n^1$, the event $A_{n-1}(x_1)$ happens, and for at
least one $x_2\in {\mathcal H}_n^2$, the event $A_{n-1}(x_2)$
happens (as in the proof of Lemma~\ref{l:C}, this follows from the continuity of a path connecting $S(0,a_n/10)$ and $\partial S(0,a_n)$; an illustration would look similar to Figure~\ref{f:renormalization}). Consequently, using \eqref{e:|H|},
\begin{equation}\label{Aest1}
\begin{split}
\P_u[A_n(0)]&\le \sum_{x_1\in {\mathcal
H}_n^1}\sum_{x_2\in {\mathcal
H}_n^2}\P_u[A_{n-1}(x_1)\cap A_{n-1}(x_2)]\\
&\le
c\left(\frac{a_n}{a_{n-1}}\right)^2\sup_{x_1\in {\mathcal
H}_n^1,\,x_2\in {\mathcal H}_n^2}\P_u[A_{n-1}(x_1)\cap
A_{n-1}(x_2)].
\end{split}
\end{equation}
Next, we will find a useful upper bound on
$\P_u[A_{n-1}(x_1)\cap A_{n-1}(x_2)]$.
  For $x_1\in
{\mathcal H}_n^1$, $x_2\in {\mathcal H}_n^2$, we have (recall the notation \eqref{L-})
\begin{multline}\label{aest2}
\P_u\left[\cap_{i=1}^2 A_{n-1}(x_i)\right] \le
\P_u\left[\omega\left(L^-_{x_1,x_2,n-1}\right)\neq
0\right]\\+\P_u\left[\cap_{i=1}^2
A_{n-1}(x_i)\cap\left\{\omega\left(L^-_{x_1,x_2,n-1}\right)=
0\right\}\right],
\end{multline}
and we will now find upper bounds on the two terms on the right-hand side.
Using Lemma~\ref{squarest} and the inequality $1-e^{-x}\le x$
for $x\ge 0$, we obtain
\begin{equation}\label{aest3}
\P_u\left[\omega\left(L^-_{x_1,x_2,n-1}\right)\neq
0\right]=1-\exp(-u \mu(L^-_{x_1,x_2,n-1}))\le cu
\frac{a_{n-1}^2}{a_{n}^{d-1}}.
\end{equation}
We now turn to the second term in~\eqref{aest2}. First we observe that since $L^+_{x_1,x_2,n-1}$ and  $L^+_{x_2,x_1,n-1}$ are disjoint sets of lines (cf.~\eqref{L+}), it follows that
\begin{equation}\label{aindep}
A_{n-1}(x_1,L^+_{x_1,x_2,n-1})\mbox{ and }A_{n-1}(x_2,L^+_{x_2,x_1,n-1})\mbox{ are independent.}
\end{equation}
We now get that
\begin{equation}\label{aest4}
\begin{split}
&\P_u\left[\cap_{i=1}^2
A_{n-1}(x_i)\cap\left\{\omega\left(L^-_{x_1,x_2,n-1}\right)=
0\right\}\right]\\
&=\P_u\left[\cap_{i=1}^2
A_{n-1}(x_i, L^+_{x_i,x_{3-i},n-1})\cap\left\{\omega\left(L^-_{x_1,x_2,n-1}\right)=
0\right\}\right]\\
&\le \P_u\left[\cap_{i=1}^2
A_{n-1}(x_i,L^+_{x_i,x_{3-i},n-1})\right]
\stackrel{~\eqref{aindep}}{=}\prod_{i=1}^2\P_u[A_{n-1}(x_i,L^+_{x_i,x_{3-i},n-1})]\\
&\stackrel{~\eqref{aincr}}{\le}
\prod_{i=1}^2\P_u[A_{n-1}(x_i)]\stackrel{~\eqref{atransinv}}{=}\P_u[A_{n-1}(0)]^2.
\end{split}
\end{equation}
Combining ~\eqref{Aest1}, ~\eqref{aest2}, ~\eqref{aest3},
~\eqref{aest4} we obtain
\begin{equation}
\label{recrel}
\begin{split}
p_n(u) &\le
c\left(\frac{a_n}{a_{n-1}}\right)^2\left(p_{n-1}(u)^2+cu
\frac{a_{n-1}^2}{a_{n}^{d-1}} \right)\\
&\le c_8 a_{n-1}^{2c_5-2}\left(p_{n-1}(u)^2+u
a_{n-1}^{2-c_5(d-1)}\right),
\end{split}
\end{equation}
where we have used the relation $a_n=a_{n-1}^{c_5}$ from \eqref{req1} and defined the constant $c_8>1$ in the last inequality. Under the hypothesis $d \geq 4$, we can use \eqref{recrel} and an algebraic manipulation to derive the key induction step in the following lemma:
\begin{lemma}[$d \geq 4$]
\label{indlemma}
For any $a_0 \ge (2c_8)^{10}$ \textup{(}cf.~\eqref{recrel}\textup{)} and $u\le 1$, if
\begin{equation}\label{pndec1}
p_n(u)\le a_n^{(2-c_5(d-1))/2},
\end{equation}
then also
\begin{equation}\label{pndec2}
p_{n+1}(u)\le a_{n+1}^{(2-c_5(d-1))/2}.
\end{equation}
\end{lemma}
\begin{proof}[Proof of Lemma~\ref{indlemma}.]Suppose $u\le 1$. If for some $n\ge 0$ inequality~\eqref{pndec1} holds, we get by inequality~\eqref{recrel},
\begin{equation*}
p_{n+1}(u) \le c_8 a_{n}^{2c_5-2}\left(p_{n}(u)^2+a_{n}^{2-c_5(d-1)}\right)\le
2c_8 a_{n}^{2c_5-c_5(d-1)}= 2c_8 a_{n+1}^{3-d}\le a_{n+1}^{3.1-d},
\end{equation*}
where the last inequality holds because $a_0 \ge (2c_8)^{10}$.
Furthermore, using that $c_5 = 3.8/3 < 2$ (cf.~\eqref{consts}), the inequality $3.1-d\le (2-c_5(d-1))/2$ can be written
as $d\ge (4.2-c_5)/(2-c_5) = 4$, finishing the proof of Lemma~\ref{indlemma}.
\end{proof}
With Lemma~\ref{indlemma}, we can now complete the proof of Proposition~\ref{keycontrol}. We set $a_0 \geq (2c_8)^{10}$, so that Lemma~\ref{indlemma} applies. Since
\begin{align*}
p_0(u) &\leq {\mathbb P}_u [ \omega \in \Omega: \omega(L_{B(x,a_0)}) \neq 0]\\
&= 1 - e^{-u \lambda(B_{d-1}(0,a_0+1))}, \text{ by Lemma~\ref{l:mubd0},}
\end{align*}
we can find $0< c_7(a_0) \le 1$ such that $p_0(u)\le a_0^{(2-c_5(d-1))/2}$ for every
$u\le c_7(a_0)$. Choosing such a $u$, an inductive application of Lemma~\ref{indlemma} yields
$$p_n (u)\le a_n^{(2-c_5(d-1))/2},\,n\ge 1,$$ completing the proof of
Proposition~\ref{keycontrol}.
\end{proof}

Proposition~\ref{keycontrol} now allows us to prove Theorem~\ref{lowintthm}.

\begin{proof}[Proof of Theorem~\ref{lowintthm}.]
Note that, if the event $\textup{Perc}_2^c$ occurs, then the
component of ${\mathcal V} \cap {\mathbb R}^2$ containing the
origin is empty or bounded, hence by Lemma~\ref{l:mubd0} and local
finiteness of $\omega \in \Omega$ delimited by a finite number of
ellipsoids. For every such ellipsoid, there is a surrounding path
following its boundary in the clockwise direction. By
concatenating the pieces of these paths running along the boundary
of the vacant component containing $0$, one obtains a closed curve
$\gamma$ contained in ${\mathcal L} \cap {\mathbb R}^2$
surrounding $0$. Hence, we have the inclusion
\begin{align}
 \label{fincalc0}
\textup{Perc}_2^c \subseteq \left\{ \begin{array}{c}
\mbox{there is a closed curve }\gamma\subset {\mathcal L}\cap
\plane \\
\mbox{surrounding } 0
\end{array}  \right\}.
\end{align}
For $k\ge 1$ let $s_k=\lceil 5 a_k / a_{k-1} \rceil$. For $k\ge 1$
and $i=1,...,s_k$ we define the points $x_{k,i}\in {\mathbb
R}_{+}$ as $x_{k,i}=a_{k-1}+a_{k-1}/10+(i-1)a_{k-1}/5$. We now
choose $a_0$ large and $u$ small so that~\eqref{keydec}
holds. Using \eqref{fincalc0}, we get that
\begin{equation}\label{fincalc0.1}
\begin{split}
& \textup{Perc}_2^c \cap \{ \omega \in \Omega: \omega(L_{S(0,a_0)})=0 \} \\
&\subseteq \left\{
\begin{array}{c}
\mbox{there is a closed curve }\gamma\subset {\mathcal L}\cap
(\plane\setminus S(0,a_0)) \\
\mbox{surrounding } 0
\end{array} \right\} \\
&\subseteq \bigcup_{k=1}^{\infty} \left\{
\begin{array}{c}
\mbox{there is a closed
curve }\gamma\subset {\mathcal L}\cap (\plane\setminus S(0,a_0)) \\
\mbox{surrounding }0 \mbox{ and
intersecting }[a_{k-1},a_k]e_1 \end{array} \right\}.
\end{split}
\end{equation}
By continuity of $\gamma$, the set on the right-hand side is included in $$\cup_{k=1}^\infty \cup_{i=1}^{s_k} A_{k-1}(x_{k,i})$$ (see Figure~\ref{f:peierls} for a schematic illustration),
\begin{figure}
\psfrag{0}[cc][cc][1.5][0]{$0$}
\psfrag{a0}[cc][cc][1.5][0]{$a_0$}
\psfrag{a1}[cc][cc][1.5][0]{$a_1$}
\psfrag{a2}[cc][cc][1.5][0]{$a_2$}
\psfrag{gamma}[cc][cc][1.5][0]{$\gamma$}
\begin{center}
\includegraphics[angle=0, width=0.8\textwidth]{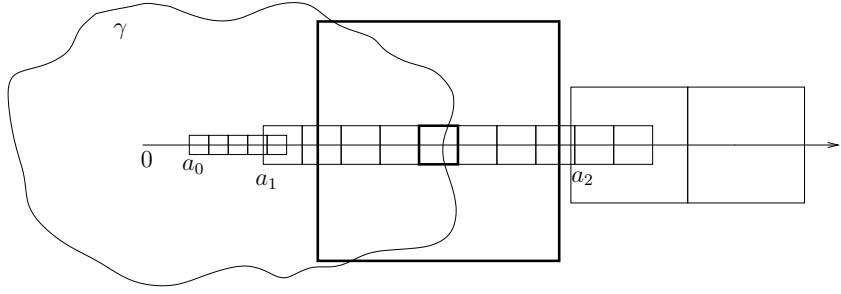}
\end{center}
\caption{A schematic illustration of the construction used in the proof of Theorem~\ref{lowintthm}. If a closed curve $\gamma$ surrounds $0$ and does not intersect $S(0,a_0)$, then one of the events $A_{k-1}(x_{k,i})$, $k \geq 1$, $1 \leq i \leq s_k$ must occur (here, $A_1(x_{2,5})$ occurs).} \label{f:peierls}
\end{figure}
hence
\begin{equation}
\label{fincalc}
\begin{split}
& {\mathbb P}_u [\textup{Perc}_2^c,\omega(L_{S(0,a_0)})=0 ]\\
&  \le \sum_{k=1}^{\infty}
\sum_{i=1}^{s_k}\P_u[A_{k-1}(x_{k,i})]\le
\sum_{k=1}^{\infty}c \frac{a_k}{a_{k-1}}\P_u[A_{k-1}(0)]\\
&\le c\sum_{k=1}^{\infty}
a_{k-1}^{c_5-1}a_{k-1}^{(2-c_5(d-1))/2}, \text{ by } \eqref{keydec}, \\
&= c \sum_{k=1}^\infty a_{k-1}^{-c_5((d-1)/2 - 1)} \leq c \sum_{k=1}^\infty a_0^{-c_5^k/2}.
\end{split}
\end{equation}
Thus,
\begin{align*}
\P_u\left[\textup{Perc}_2^c\right] &\le \P_u\left[\textup{Perc}_2^c,\,\omega(L_{S(0,a_0)})=0\right]
+\P_u[\omega(L_{S(0,a_0)})>0]
\\
&\stackrel{~\eqref{fincalc}}{\le} c a_0^{-c} + u \mu(L_{S(0,a_0)}).
\end{align*}
Choosing $a_0$ sufficiently large, and then $u$ sufficiently
small makes $c a_0^{-c}+u \mu(L_{S(0,a_0)})<1$, and
we conclude that $\P_u\left[\textup{Perc}_2^c\right]<1.$ Therefore, by the $0$-$1$-law \eqref{e:0-1-perc2},
$\P_{u}[\textup{Perc}_2]=1$ and the proof of Theorem~\ref{lowintthm} is complete.
\end{proof}

The above proof showing percolation at low intensities works for dimensions $d \geq 4$. Deciding whether or not percolation of the vacant set occurs at low intensities in dimension $3$ is currently an open problem. In order to illustrate that this case is substantially more delicate, we shall now prove that in dimension $3$, the set ${\mathcal V} \cap {\mathbb R}^2$ does not percolate. By Theorem~\ref{lowintthm}, this phenomenon provides a sharp contrast to the case $d \geq 4$.

\begin{proposition}
 \label{p:triangles}
For dimension $d=3$ and any $u>0$,
\begin{align}
 \label{e:triangles}
{\mathbb P}_u [ \textup{Perc}_2 ] = 0.
\end{align}
\end{proposition}

\begin{remark}
 By translation and rotation invariance of ${\mathbb P}_u$, Proposition~\ref{p:triangles} implies that for any fixed two-dimensional subspace $S_2$ of ${\mathbb R}^3$ and any $x \in {\mathbb R}^d$, $\mathcal V$ does not percolate along $x+S_2$, i.e. ${\mathbb P}_u [ {\mathcal V} \cap (x+S_2)  \text{ percolates}] = 0$ for any $u>0$.
\end{remark}

\begin{proof}[Proof of Proposition~\ref{p:triangles}.]
For any $a > 0$, we define the segments
\begin{align*}
 S_1^\pm(a) = \bigg\{ \pm \frac{\sqrt{3}}{2} a \bigg\} \times \left[ - \frac{a}{2}, - \frac{a}{4} \right] \times \{0\} \subset {\mathbb R}^2 \subset {\mathbb R}^3.
\end{align*}
Crucially, any cylinder intersecting both $S_1^-(a)$ and $S_1^+(a)$ covers a line in ${\mathbb R}^2$ connecting the segments $\{ - \frac{\sqrt{3}}{2} a \} \times \left[ - \frac{a}{2}, - \frac{a}{4} \right]$ and $\{ \frac{\sqrt{3}}{2} a \} \times \left[ - \frac{a}{2}, - \frac{a}{4} \right]$.
\begin{figure}
\psfrag{S1-}[cc][cc][2.5][0]{$S_1^-(a)$}
\psfrag{S1+}[cc][cc][2.5][0]{$S_1^+(a)$}
\psfrag{S2-}[cc][cc][2.5][0]{$S_2^-(a)$}
\psfrag{S2+}[cc][cc][2.5][0]{$S_2^+(a)$}
\psfrag{S3-}[cc][cc][2.5][0]{$S_3^-(a)$}
\psfrag{S3+}[cc][cc][2.5][0]{$S_3^+(a)$}
\psfrag{0}[cc][cc][2.5][0]{$(0,0)$}
\psfrag{r}[cc][cc][2.5][0]{$(a,0)$}
\begin{center}
\includegraphics[angle=0, width=0.5\textwidth]{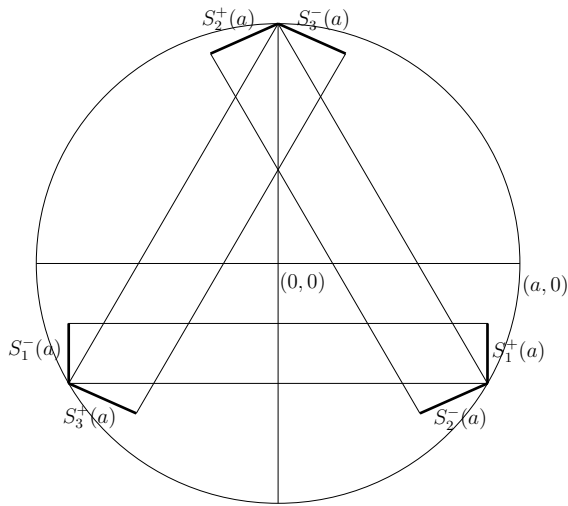}
\end{center}
\caption{An illustration of the segments $S_i^\pm(a)$, $i \in \{1,2,3\}$, defined in the proof of Proposition~\ref{p:triangles}.} \label{f:triangles}
\end{figure}
We now obtain similar rectangles $S_2^\pm(a)$ and $S_3^\pm(a)$ by rotating $S_1^\pm(a)$ by $2\pi/3$ and $4\pi/3$: For $R_{2\pi/3}$ denoting rotation of ${\mathbb R}^2$ around $0$ by the angle $2\pi/3$ in the counterclockwise direction, we set (see Figure~\ref{f:triangles} for an illustration)
\begin{align*}
 S_2^\pm(a) = R_{2\pi/3} S_1^\pm(a), \\
 S_3^\pm(a) = R_{2\pi/3} S_2^\pm(a).
\end{align*}
Then we define the event $\Delta_a$ as the event that a cylinder intersecting both $S_i^-(a)$ and $S_i^+(a)$ occurs for all $i \in \{1,2,3\}$:
\begin{align*}
\Delta_a = \bigcap_{i=1}^3 \big\{ \omega \in \Omega: \omega(L_{S_i^-(a), S_i^+(a)}) \geq 1 \big\}.
\end{align*}
Observe that
\begin{align}
\label{triangles1}
 \big\{ \{0\} \stackrel{{\mathcal V} \cap {\mathbb R}^2}{\longleftrightarrow} \infty \big\} \subseteq \Omega \setminus \Delta_a, \text{ for all } a > 0.
\end{align}
Indeed, if $\Delta_a$ occurs, then the cylinders associated to any
triple of lines in $L_{S_1^-(a), S_1^+(a)} \times L_{S_2^-(a),
S_2^+(a)} \times L_{S_3^-(a), S_3^+(a)}$ cover the lines of a
triangle in ${\mathbb R}^2$ whose interior contains $(0,0)$
(cf.~Figure~\ref{f:triangles}). For $a$ larger than some constant $c$, we can fix equally spaced points
$x_j^- = (-\sqrt{3}a/2, -a/2 + bj,0)$ on $S_1^-(a)$ and $x_j^+ =
(\sqrt{3}a/2, -a/2 + bj,0)$ on $S_1^+(a)$ for $b>0$ and $j=0,1,
\ldots, J \in {\mathbb N}$, such that the sets $(L_{\{x_j^-\},
\{x_k^+\}})_{(j,k) \in \{0,1, \ldots, J\}^2}$ are mutually
disjoint: since $d(S_1^-(a), S_1^+(a)) = \sqrt{3}a$, the ${\mathbb
R}^2$-projection of any line whose cylinder intersects both
$S_1^-(a)$ and $S_1^+(a)$ has a slope of at most a constant $c_9 \in (0, \infty)$ for $a \geq c$ (indeed, the largest possible slope of such a line converges to $1/(4 \sqrt(3))$ as $a$ tends to infinity), so it is sufficient to choose $b = 2 \sqrt{1+c_9^2}$ and $J = [a/b] \geq ca$. Then we deduce with the
help of Lemma~\ref{l:nubd} applied to balls of radius $0$ and with
$\alpha = ca$ that
\begin{align*}
 \mu(L_{S_1^-, S_1^+}) &\geq \sum_{j=0}^J \sum_{k = 0}^J \mu (L_{\{x_j^-\}, \{x_k^+\}}) \geq \sum_{j=0}^J \sum_{k = 0}^J \frac{c}{a^2} \geq c_{10} >0,
\end{align*}
where the constant $c_{10}$ does not depend on $a$. Since the sets $(L_{S_i^-(a), S_i^+(a)})_{i=1}^3$ are disjoint for $a \geq 3$, the random variables $\omega(L_{S_i^-(a), S_i^+(a)})$ are independent under ${\mathbb P}_u$ and by rotation invariance of $\mu$ all Poisson-distributed with parameter $$u \mu ( L_{S_1^-(a), S_1^+(a)} ).$$ Hence, we can deduce from the last estimate that for $a \geq c + 3$,
\begin{equation}
\label{triangles2}
\begin{split}
{\mathbb P}_u [\Delta_a] &\geq  \left( 1 - e^{- u \mu ( L_{S_1^-(a), S_1^+(a)} )} \right)^3 \\
& \geq (1- e^{- u c_{10}} )^3 =: c_{11}(u)>0,
\end{split}
\end{equation}
where $c_{11}(u)$ does not depend on $a$. We now use this estimate on the sequence $a_n = 3^n$, $n \geq 1$. Note that $-3^n/2 > - 3^{n+1}/4 + b$ for $n \geq c$, so the set of cylinders intersecting both $S_1^-(3^n)$ and $S_1^+(3^n)$ is disjoint from the set of cylinders intersecting both $S_1^-(3^{n+1})$ and $S_1^+(3^{n+1})$, and the analogous statement also holds for $S_2^\pm$ and $S_3^\pm$. The events $(\Delta_{3^n})_{n \geq c}$ are therefore independent and we obtain from \eqref{triangles1}, \eqref{triangles2} and the Borel-Cantelli Lemma that ${\mathbb P}_u [ \{0\} \stackrel{{\mathcal V} \cap {\mathbb R}^2}{\longleftrightarrow} \infty ] = 0$, hence by translation invariance ${\mathbb P}_u [ \{x\} \stackrel{{\mathcal V} \cap {\mathbb R}^2}{\longleftrightarrow} \infty ] = 0$ for any $x \in {\mathbb R}^2 \subset {\mathbb R}^3$.  Since ${\mathbb P}_u [ \textup{Perc}_2 ] \leq \sum_{x \in {\mathbb Q}^2 \subset {\mathbb R}^3} {\mathbb P}_u [\{x\} \stackrel{{\mathcal V} \cap {\mathbb R}^2}{\longleftrightarrow} \infty ]$, this implies \eqref{e:triangles} and thus completes the proof of Proposition~\ref{p:triangles}.
\end{proof}

\begin{remark}\label{r:problems}
 We conclude the article by mentioning a few of the open problems raised by the above results.
\begin{enumerate}
\item Does the vacant set percolate for small $u>0$ in dimension $d=3$? We have seen in Proposition~\ref{p:triangles} that in dimension three, the vacant set does not percolate along any fixed two-dimensional subspace. This property distinguishes the Poisson cylinder model from both standard percolation and random interlacements, where percolation of the vacant set intersected with the two-dimensional plane does occur in dimension $3$, see \cite{SS09}. Proposition~\ref{p:triangles} does not rule out a random set of exceptional planes, however, along which percolation of the vacant set may occur even in dimension $3$ (see \cite{BS98} for the occurrence of such a phenomenon).
\item In the regime where the vacant set percolates, is the infinite component unique, as is the case in classic percolation (cf. \cite{GG99}, Theorem 8.1) and in random interlacements (cf. \cite{T08})?
\item Is the set of Poisson cylinders connected, as is the random interlacement (cf.~\cite{Szn09}, (2.21))?
\item What is the approximate value of $u_*(d)$ (for random interlacements, this question is studied in \cite{ASz10})?
\item What is the value of $\theta(u_*(d))$?
\end{enumerate}
\end{remark}

\bibliographystyle{plain}
\bibliography{tubes}

\end{document}